\documentclass[reqno,12pt]{amsart}
\newcommand{\be}{\begin{eqnarray}}
	\newcommand{\ee}{\end{eqnarray}}
\newcommand{\beq}{\begin{equation}}
	\newcommand{\eeq}{\end{equation}}
\newcommand{\ben}{\begin{eqnarray*}}
	\newcommand{\een}{\end{eqnarray*}}

\newtheorem{theorem}{Theorem}
\newtheorem*{theoremA}{Theorem A}
\newtheorem*{theoremB}{Theorem B}

\newtheorem{corollary}[theorem]{Corollary}

\newtheorem{defi}{Definition}

\newtheorem{lemma}[theorem]{Lemma}

\newtheorem{proposition}[theorem]{Proposition}

{\catcode`\@=11\global\let\AddToReset=\@addtoreset
	\AddToReset{equation}{section}
	
	\AddToReset{theorem}{section}

\usepackage{graphicx}
	\usepackage{color}

	\newcommand{\RR}{\mathbb R}	
	\newcommand{\R}{\mathbb R}
\newcommand{\seta}{\eta^-}
\newcommand{\ueta}{\eta^+}
\begin{document}
\title[A new approach to the study of elliptic semilinear equations]{A new approach to the study of elliptic semilinear equations }%\thanks{This research was supported by  FONDECYT- 1170665.}
\author{Pilar Herreros}
\address{Departamento de Matem\'atica, Pontificia
        Universidad Cat\'olica de Chile,
       Casilla 306, Correo 22,
        Santiago, Chile.}
\email{\tt pherrero@mat.uc.cl}

%%%%%%%%%%%%%%%%%%%%%%%%%%%%%%%%%%%%%%%%%%%%%%%%%%%%%%%%%%%%%%%%%%%%%%%%

\begin{abstract}

In this paper we define a new operator $J$ for the study of
 $$ \Delta u +f(u)=0,\quad x\in \R ^N, N> 2. $$
Using $J$ we can easily see some qualitative properties of the solutions, for example we can determine how many times $u$ changes sign, which are the values of the local maxima and minima, and where $u$ changes concavity.

We also use this functional to construct nonlinearities $f$ such that this problem has at least two bound state solutions that change sign $j$ times, for $j=1,\dots,k-1$. And another $f$ such that this problem has a unique ground state solution, and at least two bound state solutions that change sign one time.
\end{abstract}

\maketitle

\section{Introduction and main results}

The main goal of this paper is to introduce a new approach for the study of radial solutions of
\begin{eqnarray}\label{pde}
\begin{gathered}
 \Delta u +f(u)=0,\quad x\in \R^N, N> 2, \\
 \lim\limits_{|x|\to\infty}u(x)=0.
\end{gathered}
\end{eqnarray}

Any nonconstant solution to \eqref{pde} is called a bound state solution. Bound state solutions such that $u(x)\geq0$ for all $x\in\mathbb R^N$, are referred to as a first bound state solution, or a ground state solution. A solution that changes sign $k-1$ times will be referred to as a $k^{th}$ bound state.

\medskip

We will work under the following assumptions on the nonlinearity $f$, that determine the regularity and general shape of $f$:

\begin{enumerate}
 \item[$(H_1)$]\begin{itemize} \item $f\in C[0,\infty) \cap C^1(0,\infty)$, except possibly in a finite number of  points where it is not differentiable.
\item $f(0)=0$ and there exist $b\ge 0$ such that $f(s)>0$ for $s>b$,
	      $f(s)\le 0$ for $s\in[0,b]$ and moreover
	      $f(s)<0$ on $(0,\epsilon)$ for some $\epsilon>0$.
\item We extend $f$ to $\R$ by $f(s)=-f(-s)$.
\item Using the notation
\begin{equation*}
 F(s) = \int_0^s f(t) dt,
\end{equation*}
 there exists a unique $\beta\ge b$ such that $F(\beta)=0$.
 \end{itemize}
\end{enumerate}

In $1981$  Gidas, Ni and  Nirenberg in \cite{gnn} proved that all ground state solutions to this problem are radially symmetric, and many other symmetry results have followed (see, for example, \cite{bn,fl,ln}). In this work we will restrict ourselves to radial solutions. Therefore we will work with the  radial version of \eqref{pde}, that is
\begin{eqnarray}\label{eq2}
\begin{gathered}
u''+\frac{N-1}{r}u'+f(u)=0,\quad r>0,\quad N> 2,\\
u'(0)=0,\quad \lim\limits_{r\to\infty}u(r)=0,
\end{gathered}
\end{eqnarray}
where   $'$ denotes differentiation with respect to $r$.

We will approach this by understanding the behaviour of solutions of the initial value problem
\begin{eqnarray}\label{ivp}
\begin{gathered}
u''+\frac{N-1}{r}u'+f(u)=0,\quad r>0,\quad N> 2,\\
u(r_0)=\alpha,\quad u'(r_0)=\bar\alpha
\end{gathered}
\end{eqnarray}
for some $r_0\ge 0,\ \alpha>0$ and $\bar\alpha\in\mathbb R$.

Given a solution $u$ of \eqref{ivp}, we introduce the operator
$$J(u(r))=-\frac{u'(r)}{r}. $$
Different values of $J$, when compared to $f(u(r))$, mark different points in the solution $u$ and help to understand the behaviour of such solutions. For example, from de values of $J$, we can determine how many times $u$ changes sign, which are the values of the local maxima and minima, and where $u$ changes concavity.  We refer to Section \ref{sec J} for more details.\\

The study of uniqueness (or multiplicity) of solutions to \eqref{eq2} involves the comparison of solutions, usually solutions that have initial conditions close to each other.
As part of this new approach, we revisit some of the operators used for the study of solutions and write them in terms of $J$. For this we will require $f$ to satisfy a subcriticality condition:
\begin{enumerate}
 \item[$(H_2)$] $(F/f)'(s) > (N-2)/(2N)$ for all $s>\beta$;
\end{enumerate}
We use this condition to prove that the operators $J$ corresponding to different solutions that change sign compare well. They do not intersect unless they have negative energy. See Proposition \ref{Prop J comparan}.\\

As an application of this approach, in Section $5$ we construct functions $f$ that have multiple $k^{th}$-bound states.
We will approach this problem using a function $f$ defined by parts as
 \begin{eqnarray}\label{fmu}
f_\mu(s)=\begin{cases}
f_1(s) &  s\leq \alpha_1=\alpha_*^k+\epsilon\\
L(s) & \alpha_1\leq s\leq \alpha_1+\epsilon\\
\lambda^2 f_2\left(\frac{s}{\mu}\right) &  s\geq \alpha_1+\epsilon
\end{cases}
\end{eqnarray}
where $f_1$ satisfies $(H_1)$ and $(H_2)$, $L(s)$ is the line from $(\alpha_1, f_1(\alpha_1))$ to $(\alpha_1+\epsilon, \lambda^2f_2(\alpha_1+\epsilon))$, and $\alpha_1$ and $\epsilon$ are given and will be chosen later. We will further assume
\begin{itemize}
 \item[$(H_3)$]  $(sf'/f)(s)$ decreasing for all $s>b$, with $(sf'/f)(\beta)<\frac{N}{N-2}$.
\item[$(H_4)$] There is an initial condition $\alpha_*^k$ such that the solution to
\begin{eqnarray}
\begin{gathered}
u''+\frac{N-1}{r}u'+f_1(u)=0,\quad r>0,\quad N> 2,\\
u(0)=\alpha_*^k,\quad u'(0)=0,
\end{gathered}
\end{eqnarray}
is a $k^{th}$-bound state solution.
\end{itemize}
In the first theorem we prove multiplicity of all $j$-bound state solutions, for $j=0,1\dots,k$.  For this we will consider $f_2$ that satisfy
\begin{itemize}
\item[$(H_5)$] there is an initial condition $\widehat\alpha>\alpha_*^k$ such that the solution to
\begin{eqnarray}
\begin{gathered}
v''+\frac{N-1}{r}v'+f_2(u)=0,\quad r>0,\quad N> 2,\\
v(0)=\widehat\alpha,\quad v'(0)=0,
\end{gathered}
\end{eqnarray}
reaches $v(r_0)=0$ with $v'(r_0)<0$.
\end{itemize}

\begin{theoremA}\label{all bound states}
Assume that $f_1$, $\alpha_*^k$, $f_2$ and $\widehat\alpha$ satisfy the assumptions above and $0<\epsilon<\min\{\beta/4, (\widehat\alpha-\alpha^k_*)/2\}$.
There is a positive constant $\bar \mu(\epsilon)$ such that, for any  $\mu>\bar \mu$ there is a $\bar \lambda(\mu)$ such that for all $\lambda>\bar \lambda$, problem \eqref{pde} with $f_\mu$ given by \eqref{fmu} has at least two bound state solutions that change sign $j$ times, for $j=1,\dots,k-1$.
\end{theoremA}

The second theorem is perhaps more surprising, it shows that sometimes there is multiplicity of $2^{nd}$-bound states but only one ground state solution. For this we will consider a function of the form
 \begin{eqnarray}\label{fa}
f_a(s)=\begin{cases}
f_1(s) &  s\leq \alpha_1=\alpha_*^k+\epsilon\\
L(s) & \alpha_1\leq s\leq \alpha_1+\epsilon\\
\lambda^2(u+a)^p &  s\geq \alpha_1+\epsilon
\end{cases}
\end{eqnarray}
with $f_1$ that satisfies $(H_1)-(H_4)$, and $p>\frac{N+2}{N-2}$.

\begin{theoremB}\label{Teo 2 bound states}
Assume that $f_1$, $\alpha_*^{k}$, $\alpha_*^{k+1}$ and $f_a$ as given by \eqref{fa}, and $0<\epsilon<\beta/4$.
Then there are constants $a$ and a $\bar \lambda(a)$ such that for all $\lambda>\bar \lambda$, problem \eqref{pde} with $f=f_a$ has at least two bound state solutions that change sign $k$ times.
Moreover, there are no bound state solutions that change sign $j$ times, for $j<k$, with initial condition $\alpha>\alpha_*^k$.
\end{theoremB}

\begin{corollary}\label{Corollary uniqueness}
Assume that $f_1$, $\alpha_*^1$, $\alpha_*^2$ and $f_a$ as given by \eqref{fa}, and $0<\epsilon<\beta/4$. and $f_1$ has a unique ground state solution with initial condition in $(\beta, \alpha_*^2)$. Then there are positive constant $a$, $\epsilon$ and $\bar \lambda$ such that for all $\lambda>\bar \lambda$, problem \eqref{pde} with $f=f_a$ has a unique ground state solution, and at least two bound state solutions that change sign one time.
\end{corollary}

\medskip

\bigskip

\section{Some properties of the solutions of the initial value problem}\label{prel}

The aim of this section is to establish several properties of the solutions to the initial value problem \eqref{ivp}.
Since $f$ is continuous, problem \eqref{ivp} has a solution defined for all $r\ge 0$ for any $\alpha>\beta$  but it might not be uniquely defined. It is straight forward to see that unique extendibility can be lost only if $u$ reaches a double zero. In this case, we will extend the solution as $0$, and consider it a bound state solution.

By standard theory of ordinary differential equations, the solution depends continuously on the initial data in any compact subset of its domain of definition.

Let us set
$$Z_1(\alpha)=\sup\{r>0\ |\ u(s,\alpha)>0\mbox{ and }u'(s,\alpha)<0\ \mbox{ for all }s\in(0,r)\}$$
and define
\begin{eqnarray*}
{\mathcal N_1}&=&\{\alpha\in[\beta,\gamma_*)\ :\ u(Z_1(\alpha),\alpha)=0\quad\mbox{and}\quad u'(Z_1(\alpha),\alpha)<0\}\\
{\mathcal G_1}&=&\{\alpha\in[\beta,\gamma_*)\ :\ u(Z_1(\alpha),\alpha)=0\quad\mbox{and}\quad u'(Z_1(\alpha),\alpha)=0\}\\
{\mathcal P_1}&=&\{\alpha\in[\beta,\gamma_*)\ :\ u(Z_1(\alpha),\alpha)>0\},
\end{eqnarray*}
where $\beta$ is as defined in $(H_1)$.  If $Z_1(\alpha)=\infty$ we consider $u(Z_1(\alpha),\alpha)=\lim_{r\to\infty}u(r,\alpha)$ that can only be $0$ or $b$, thus $\alpha\in {\mathcal G_1}$ or ${\mathcal P_1}$ respectively.

We now extend these definitions by induction for $k\ge2$. If ${\mathcal N_{k-1}}\not=\emptyset$, we set
\ben
T_{k-1}(\alpha)=&\sup\{r\in(Z_{k-1}(\alpha),D_\alpha)\ :\ (-1)^ku'(r,\alpha)\le 0\},
\een
and
\begin{eqnarray*}
Z_k(\alpha)=\sup\{r>T_{k-1}(\alpha)\ |\ (-1)^ku(s,\alpha)<0\mbox{ and }(-1)^ku'(s,\alpha)>0\ \\
\mbox{ for all }s\in(T_{k-1}(\alpha),r)\},
\end{eqnarray*}
 if $T_{k-1}(\alpha)=\infty$, we set $Z_k(\alpha)=\infty$.
\begin{figure}[h]
\begin{center}
 \includegraphics[keepaspectratio, width=13cm]{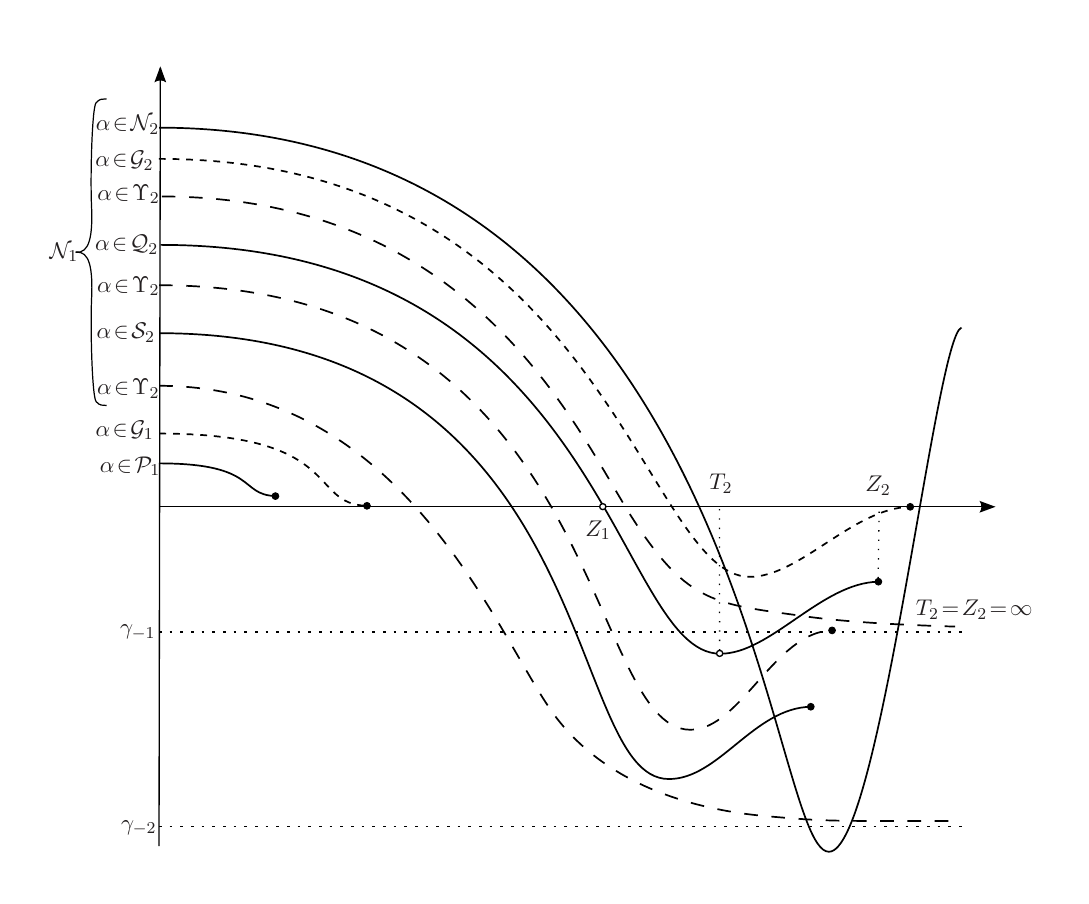}
 \end{center}
 \caption{Solutions of \eqref{ivp} with initial condition in these sets}
\end{figure}

We now define
\begin{eqnarray*}
{\mathcal N_k}&=&\{\alpha\in\mathcal N_{k-1}\ :\ u(Z_k(\alpha),\alpha)=0\quad\mbox{and}\quad (-1)^ku'(Z_k(\alpha),\alpha)>0\},\\
{\mathcal G_k}&=&\{\alpha\in\mathcal N_{k-1}\ :\ u(Z_k(\alpha),\alpha)=0\quad\mbox{and}\quad u'(Z_k(\alpha),\alpha)=0\},\\
{\mathcal P_k}&=&\{\alpha\in\mathcal N_{k-1}\ :\ (-1)^ku(Z_k(\alpha),\alpha)<0\}.
\end{eqnarray*}

We will use some known properties of solutions of \eqref{ivp} and the sets defined above, collected in the following Lemma. For a proof see for instance \cite{cghh15} and references therein.
\medskip

\begin{lemma}\label{Lema Preliminares}
Assume that $f$ satisfies $(H_1)$ and let $k\in\mathbb N$.
\begin{enumerate}
\item[(i)]
The sets ${\mathcal N_k}$ and ${\mathcal P_k}$ are open in $[\beta,\gamma_*)$.
\item[(ii)]
The boundary of $\mathcal G_k\cup\mathcal P_k$ is contained in $\bigcup_{i=1}^k\mathcal G_i$.
\item[(iii)] Any solution $u$ of \eqref{ivp} has at most a finite number of sign changes.
\end{enumerate}

\end{lemma}

%%%%%%%%%%%%%%%%%%%%%%

%%%%%%%%%%%%%%
We will also need a bound on $r$ and $ r|u'( r)|$ when the solution crosses a given interval.

\begin{lemma}\label{epsilon}
Let $v$ be the solution to the initial value problem
\begin{eqnarray}\label{ge}
\begin{gathered}
v''+\frac{N-1}{r}v'+g(v)=0,\quad r>r_\delta,\\
v(r_\delta)=\bar\alpha +\delta,\quad v'(r_\delta)=v'_\delta<0
\end{gathered}
\end{eqnarray}
where $g$ is a positive continuous function defined in $[\bar\alpha,\bar\alpha+\delta]$ and $r_\delta>0$. Let $\bar r$ be defined by  $v(\bar r)=\bar\alpha$, $||g||_+ = \max \{g(s) : s \in [\bar\alpha,\bar\alpha +\delta]\}$. Consider $\frac{\zeta}{N-2}>\delta$ and  $B>1$ given by $\delta \frac{B ^{N-2}}{B ^{N-2}-1}=\frac{\zeta}{N-2}$.

If  $\zeta<r_{\delta}|v'(r_\delta)|$  then $\bar r<B  r_\delta$ and $$\frac{\zeta}{B ^{N-2}}\le \bar r|v'(\bar r)|\le r_\delta|v'(r_\delta)|+\frac{B^N-1}{N}||g||_+r_\delta^N .$$
\end{lemma}

\begin{proof}
Since $v$ is a solution of \eqref{ge}, we have that for $r\in[r_\delta,\bar r ]$
\begin{equation}\label{m0}
	r^{N-1}v'(r)= r_\delta^{N-1}v'(r_\delta) -\int_{r_\delta}^r t^{N-1}g(v(t))dt < r_\delta^{N-1}v'(r_\delta)
\end{equation}
dividing by $r^{N-1}$ and integrating over $[r_\delta,\bar r ]$ we get
\begin{equation*}
	\delta= v(r_\delta)-v(\bar r )> \frac{r_\delta^{N-1}|v'(r_\delta)|}{N-2}\left( \frac{1}{r_\delta^{N-2}} -  \frac{1}{\bar r ^{N-2}}\right)
\end{equation*}
implying
\begin{equation}\label{m1}
	\delta> \frac{r_\delta |v'(r_\delta)|}{N-2}\left( 1 -  \frac{r_\delta^{N-2}}{\bar r^{N-2}}\right) > \frac{\zeta}{N-2}\left( 1 -  \frac{r_\delta^{N-2}}{\bar r^{N-2}}\right)>\delta \frac{B ^{N-2}}{B ^{N-2}-1}\left( 1 -  \frac{r_\delta^{N-2}}{\bar r^{N-2}}\right).
\end{equation}
Thus
$\bar r < B  r_\delta$. Moreover, from \eqref{m0},
$$\bar r|v'(\bar r)|\ge \Bigl(\frac{r_\delta}{\bar r}\Bigr)^{N-2}r_\delta |v'(r_\delta)|\ge \frac{1}{B ^{N-2}}\zeta.$$
Finally, using again \eqref{m0} and since $r_\delta<\bar r<Br_\delta$, we obtain
$$\bar r|v'(\bar r)|\le r_\delta|v'(r_\delta)|+\frac{B^N-1}{N}||g||_+ r_\delta^N \
$$

\end{proof}

\section{The operator $J$}\label{sec J}

In this section we will define the operator $J$, and prove some basic facts and how it can give information on a solution if the initial value problem \ref{ivp}. In this section we only assume that $f$ satisfies $(H_1)$.

Let us set $\rho_0=0$,
$$\rho_i(\alpha):=\sup\{r>\rho_{i-1}\ | \ (-1)^i u'(s,\alpha)>0\ \mbox{ for all }s\in(\rho_{i-1},r)\},$$
and $m_i=u(\rho_i(\alpha))$.
Note that if $u(\rho_i)=0$ with $\rho_i<\infty$, we can extend the solution as $u=0$ and thus it is a bound state solution.

 Since solutions $u(r)$ are monotone decreasing in $[0,\rho_1(\alpha)]$, we can study instead its inverse $r(s)$ for $s\in[m_1, \alpha]$ which satisfies the equation
\begin{equation}\label{eq-r}
  r''(s)-\frac{N-1}{r(s)}(r'(s))^2 - f(s)(r'(s))^3=0.
\end{equation}

\begin{defi}
  For each solution $u(r)$ with $u(0)=\alpha$ and $u'(0)=0$, and its inverse $r(s)$  we define
  \begin{equation}\label{J-def}
    J(s)= \frac{|u'|}{r}= \frac{-1}{r(s)r'(s)}.
  \end{equation}
  \end{defi}
   It follows directly from equation \eqref{eq-r} that $J$ satisfies
   \begin{equation}\label{J-prima}
    J'= \frac{1}{r^2}\left(N-\frac{f}{J}\right)=\frac{N}{Jr^2} \left(J-\frac{f}{N}\right),
  \end{equation}
with $J(\alpha)=\frac{f(\alpha)}{N}$ and $J'(\alpha)= \frac{f'(\alpha)}{N+2}< \frac{f'(\alpha)}{N}$. Thus $J>\frac{f}{N}$ for $s$ in some interval $(\alpha-\epsilon, \alpha)$.

\begin{figure}\label{Fig 1J}
\includegraphics[keepaspectratio, width=8cm]{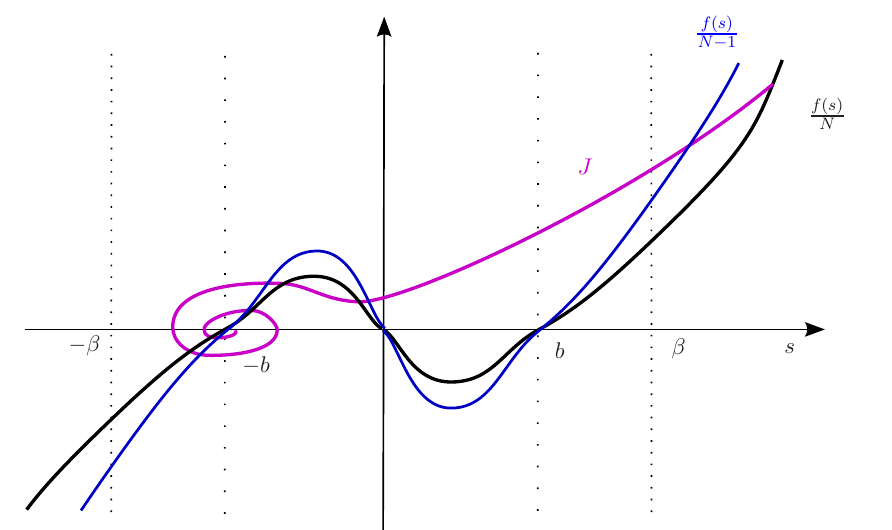}
\caption{$J$ of a solution with $\alpha\in\mathcal P_1$}
\end{figure}

 After $\rho_1(\alpha)$ the solution $u$ will be monotonously increasing in $[\rho_1(\alpha), \rho_2(\alpha)]$, we can study instead its inverse $\bar r(s)$ for $s\in[m_1, m_2]$. In this interval we define
   \begin{equation}\label{J-def-bar}
    \bar J(s)= \frac{-1}{\bar r(s)\bar r'(s)}
  \end{equation}
and note that $\bar J(s)$ corresponds to $\frac{-u'}{r}$. In this way we can define $\mathcal J(r)$ along all the solution using $r(s)$ or $\bar r(s)$ when appropriate.
 Note also that in both cases
 \begin{equation}\label{J(r)-def}
  \mathcal J(r)= \frac{-u'(r)}{r}= J(u(r))\  \mbox{ or } \ \bar J(u(r))
 \end{equation}
 thus $\mathcal J$ is differentiable in all $\R$.

Note also that if $v$ is the solution to \eqref{ivp} with $v(0)=-\alpha$, then $\bar J_u(s)=-J_v(-s)$, thus all the analysis made for $J$ is also valid for $\bar J$ with the appropriate sign changes.\\

In the next subsection we will prove that $J$ spirals counterclockwise, and from their graphs we can identify ground states, bound states, and how many times solutions cross $0$, among other things. First we will prove some basic facts.

\begin{lemma}\label{J-basicos}
  For any solution $r(s)$ on $[u(\rho_1(\alpha)), \alpha]$,  $J(s)$ and $r(s)$ satisfy
  \begin{enumerate}
    \item[(i)] $J'(s)>0$ if and only if $J(s)>\frac{f(s)}{N}$.
    \item[(ii)] $r''(s)<0$ if and only if $J(s)>\frac{f(s)}{N-1}$.
    \item[(iii)] $J(s)=0$ when $s=\rho_1(\alpha)$, and at this point $f(\rho_1(\alpha))\leq 0$.
    \item[(iv)] If for some $\bar s$,  $J(\bar s)=\frac{f(\bar s)}{N}\neq 0$, with $J(s)> \frac{f(\bar s)}{N}$ for $s>\bar s$, then $f'(\bar s)<0=J'(\bar s)$.

  \end{enumerate}
\end{lemma}

\begin{proof}Follows directly from equation \eqref{J-prima} and the fact that \eqref{eq-r} can be written as
$$r''(s)=(r'(s))^3 \left(f(s)-(N-1)J(s)\right).$$
\end{proof}

\begin{figure}\label{Fig 1Ju}
\includegraphics[keepaspectratio, width=\linewidth]{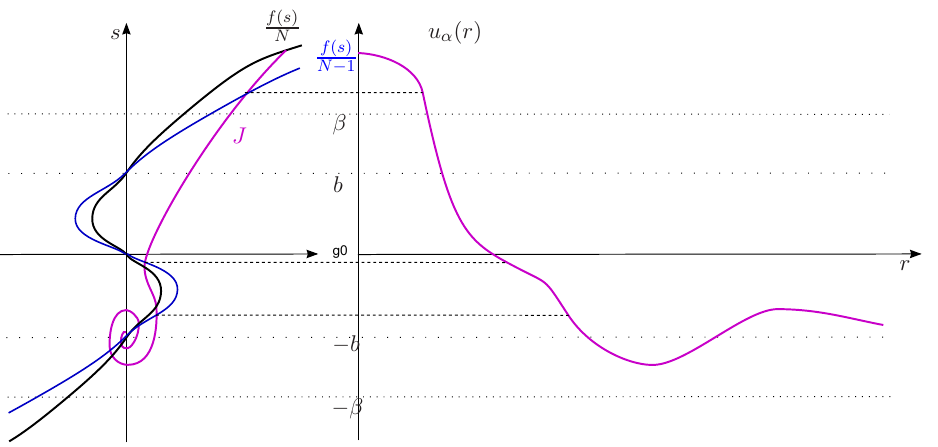}
\caption{$J(s)$ and $u(r)$ of a solution with $\alpha\in\mathcal P_1$}
\end{figure}

\subsection{General behavior of $J$ and $u$}\mbox{}\\

We will study some facts about the behaviour of $J$ and how it relates to the behaviour of $u$. We can summarize this relation in the following proposition.

\begin{proposition}\label{Prop Shape J}
For each solution $u$ of \eqref{ivp}, with $u(0)=\alpha$, the graph of the functionals $J$ and $\bar J$, or equivalently the curve $(u(r), \mathcal J(r))$ for $r\geq 0$, satisfies the following.
 \begin{enumerate}
\item[(i)] It starts at $\left(\alpha, \frac{f(\alpha)}{N}\right)$, and spirals inwards, counterclockwise, without self intersections.
\item[(ii)] It will first rotate around the origin, intersecting the $s$-axis outside $(-\beta, \beta)$. In this process, it crosses the line $s=0$ $k$ times, $0 \leq k<\infty$, indicating that $\alpha\in\mathcal N_k$  and $\alpha\notin\mathcal N_{k+1}$.
\item[(iii)] After crossing $s=0$ for the $k^{th}$ time, it rotates around $(-1)^kb$, crossing the $s$-axis $j$ times, $0 \leq j\leq\infty$.
If $k=0$, we count the initial point as $j=1$ and add the crossings after that.
     \begin{enumerate}
     \item If $j=0$, then $J(-b)=0$ (or $\bar J(b)=0$), $u$ tends asymptotically to $\pm b$, and $\alpha\in\mathcal P_{k+1}$.
     \item If $j=1$ then either $J(-b)=0$ (or $\bar J(b)=0$), $u$ tends asymptotically to $\pm b$, and $\alpha\in\mathcal P_{k+1}$, or $J(0)=0$ and $\alpha\in\mathcal G_{k+1}$, i.e. $u$ is a $k$-th bound state solution.
     \item If $j\geq 2$, $\alpha\in\mathcal P_{k+1}$ and after $j$ intersections $J$ ends in $J(\pm b)=0$ (or $\bar J(\pm b)=0$) and $u$ tends asymptotically to $\pm b$. If $j=\infty$ it spirals around $\pm b$ forever.
     \end{enumerate}
\end{enumerate}
\end{proposition}

\begin{figure}\label{Fig variasJu}
\includegraphics[keepaspectratio, width=\linewidth]{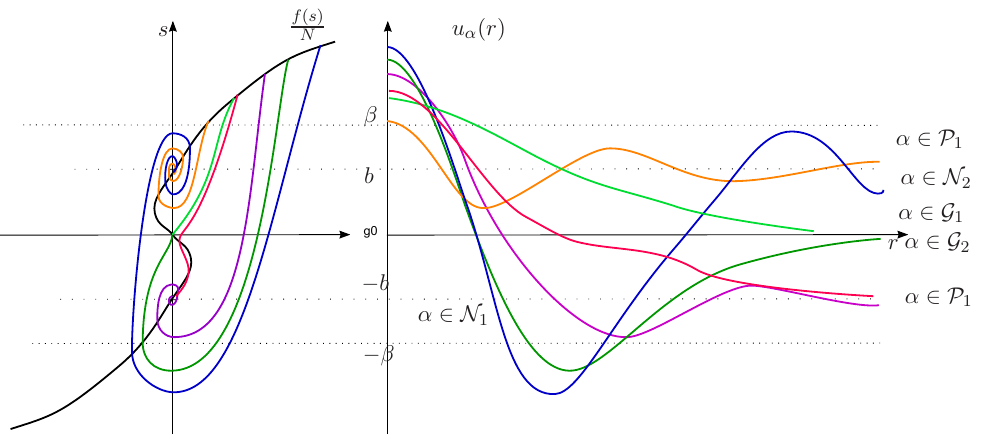}
\caption{$\mathcal J(s)$ and $u(r)$ for different initial conditions}
\end{figure}

\begin{proof}
First, we will consider the classical energy functional $I(r)=\frac{|u'|}{2}+F(u)$, now in terms of $s$.

The energy $I(s)= \frac{ r^2 J^2}{2}+F$ has $I'(s)=(N-1)J(s)$, so $I(u(r))$ is decreasing with respect to $r$.

Also, if for $r_1< r_2$ we have $u(r_1)=u(r_2)$ and $\mathcal J(r_1)=\mathcal J(r_2)$ then
$$I(s(r_1))= \frac{ r_1^2 {\mathcal J}^2(s)}{2}+F(s) <\frac{ r_2^2 {\mathcal J}^2(s)}{2}+F(s) =I(s(r_2)), $$
a contradiction. Therefore $(u(r),\mathcal J(r))$ has no self intersections, proving $(i)$.

 \bigskip

Since $I(s)=J^2r^2(s)+F(s)< F(\alpha)$, $J$ must reach $0$ in $(-\alpha,\alpha)$. Moreover, by Lemma \ref{J-basicos}$(iii)$ it must reach $0$ when $f\leq 0$, therefore it has three options:
\begin{enumerate}
\item $J(s)=0$ for some $s\in(0,b]$, then $\alpha\in \mathcal P_1$.
\item $J(0)=0$, then $u$ is a ground state solution,  $\alpha\in \mathcal G_1$.
\item $J(0)>0$, then  $\alpha\in \mathcal N_1$.
\end{enumerate}
We will continue the study of the third case.

After $J$ reaches $0$, it has to reach $J(s)=0$ with $s\in (-\alpha,-b]$.

If $J(-b)=0$ then $u''=u'=0$ and by unique continuation $r=\infty$, therefore $u$ decreases asymptotically to $-b$.

If $J(m)=0$ with $m\in (-\alpha,-b)$,  $u'(r(m))=0$ with $u''(r(m))=-f(m)>0$, a minimum,  and we use $\bar r$   to study the solution going up.
Since $\bar J>f/N$, $\bar J'<0$ until it reaches $\bar J=f/N$ with $\bar J'=0<f'/N$, therefore they cross and $\bar J'>0$ there after.

As in Lemma \ref{J-basicos} $(iv)$, if for some $\bar s$,  $\bar J(\bar s)=\frac{f(\bar s)}{N}$ again, then $f'(\bar s)<0$. Therefore $\bar J<f/N$ until either $s=0$ or $\bar J=0$, since $f'>0$ or $f>0$ for $s<0$.

Therefore $\bar J$ has three options:
\begin{enumerate}
\item $\bar J(s)=0$ for some $s\in [-b,0)$, then $\alpha\in \mathcal P_2$.
\item $\bar J(0)=0$, then $u$ is a ground state solution,  $\alpha\in \mathcal G_2$.
\item $\bar J(0)<0$, then $\alpha\in \mathcal N_2$.
\end{enumerate}
For solutions in $\mathcal N_2$ we can repeat the same argument, until the solution is in $\mathcal G_k$ or $\mathcal P_k$.
This process ends in finite steps since, by Lemma \ref{Lema Preliminares}$(iii)$, there are no $\alpha\in \mathcal N_k$ for all $k\in \mathbb{N}$.

\medskip
$(iii)$ Solutions in $\mathcal P_k$:

We will write the argument for a solution in $\mathcal P_1$ for simplicity, but it is the same for any  $\mathcal P_k$.

If $J(b)=0$ then $u''=u'=0$ and by unique continuation $r=\infty$, therefore $u$ decreases asymptotically to $b$.

If $J(m)=0$ for $m\in(0,b)$, $u$ has a minimum and we use $\bar r$ to study the solution going up. Since $\bar J>f/N$, $\bar J'<0$ until it reaches $\bar J=f/N$ with $\bar J'=0<f'/N$, therefore they cross and $\bar J'>0$ there after with $\bar J<f/N$ until $J=0$, with $f>0$ ($s\geq b$).

If $\bar J(b)=0$  $u$ increases asymptotically to $b$. If  $\bar J(M)=0$ for $M\in(b,\alpha)$, $u$ has a maximum and we use $ r$ to study the solution going down.  Since $J<f/N$, $J'>0$ until it reaches $J=f/N$ with $J'=0>f'/N$, therefore they cross and $J'>0$ there after with $J>f/N$ until $J=0$, with $f<0$ ($s\leq b$).

Repeating this argument, $J$ will spiral inward around $b$, without self intersections, until either $J$ ends in $J(b)=0$, thus $u$ tends asymptotically to $b$, or forever thus $u$ oscillates forever with decreasing amplitude.
\end{proof}

We finish this section with the observation that the sets $\mathcal N_k$ are strictly nested, they are nested by definition and the following lemma shows that $\mathcal N_k \neq \mathcal N_{k+1}$ if they are not empty. This is done by showing that solutions close to a bound state solution that cross $0$ will have negative energy before reaching $J=0$.

\begin{lemma}\label{Lema I<0}
 Let $\alpha_*^k$ be a $k^{th}$-bound state solution then there is $\epsilon >0$ such that if $\bar\alpha \in [\alpha_*^k-\epsilon,\alpha_*^k+\epsilon]\cap \mathcal N_k$ then $|u(\bar\alpha, r)|<\beta $ for all $r>Z_k$.
\end{lemma}

\begin{proof} We will work with an odd $k$, the even case differ only in some signs. Suppose the lemma is not true, then there exists a sequence $\alpha _i$ such that $\alpha _i \rightarrow \alpha_*^k$ and  $u( \alpha_i, \bar r_i)= -\beta $ for some $\bar r_i$.

   Let $H(s)= r^{2(N-1)}I(s)$, then for $s\in(-\beta,\beta)$ $$H'(s)=r^{2(N-1)}(N-1)J(s)-2(N-1)\frac{ r^{2(N-1)-2}}{J(s)}\left(\frac{ r^2 J^2}{2}+F\right) = -2(N-1) r^{2(N-1)-2}\frac{F(s)}{J(s)} >0.$$

 For $\alpha =\alpha_*^k$, since $H_*(s)>0$  and decreasing (backwards) the limit $\lim_{s\rightarrow 0}H_*(s)=L$ exists and is greater or equal $0$.
By continuity and using the fact that  $H_i'(s)>0 $ if $\beta <s<-\beta$ , we have that given $\epsilon>0$ there is $s^*>0$ such that  for  $\alpha_i$, with  large $i$,  $H_i(s)< L+\epsilon$ for all $s \in [-\beta,  s^*]$.

 As, for  $\alpha =\alpha^*$,  $\lim_{s\rightarrow 0}I_*(s) =0$ we may choose $s^*>0$ so that we also have $I_i(s)<\epsilon $ for  $\alpha_i$, with  large $i$,  for all $s \in [-\beta,  s^*]$.

 Let $F_1>0$, $F_2$  such that $-F_2 \leq F(s)\leq -F_1$ if  $s \in [-\beta/2,  -\beta/4]$.  From $I_i(s)<\epsilon $  we obtain  $ r_i^2 J_i^2 < 2 F_2 $ for  $s \in [-\beta/2,  -\beta/4]$ if $i$ is large.
So $$L+\epsilon >H_i(-\beta/4) = H_i(-\beta/2)+ \int_{-\beta/2}^{-\beta/4} -2(N-1) r^{2(N-1)-2}\frac{F(s)}{J_i(s)} ds
 $$ $$\geq  H_i(-\beta/2)+  (r_i(-\beta/4))^{2(N-1)} \frac{F_1}{F_2}\beta/4.$$

 If $\lim_{s\to 0}r_*(s)=\infty$ then $r_i(-\beta/4)> r_i(0)\to \infty$, by continuity of the solutions, and the  last term tends to infinity as $i$ is large. A contradiction.

 If $r_*(0)=r^*<\infty$,  since for  $\alpha =\alpha^*$,   $\lim_{s\rightarrow 0}J_*(s) =0$ we must have $L=0$ and, by continuity of the solutions,
  $$\epsilon \geq (r_i(-\beta/4))^{2(N-1)} \frac{F_1}{F_2}\beta/4 \geq ( r_i(0))^{2(N-1)} \frac{F_1}{F_2}\beta/4\to (r^*)^{2(N-1)} \frac{F_1}{F_2}\beta/4$$
 that gives a contradiction if we choose $\epsilon$ small enough.

 \end{proof}

\subsection{ Behaviour of $J/f$}\label{Sec J/f} \mbox{ }

We want to study the behavior of $\frac{J}{f}$, for this we note that it satisfies $\frac{J}{f}(\alpha)=1/N$,  $\left(\frac{J}{f}\right)'(\alpha)= \frac{-2}{N(N+2)}\frac{f'}{f}(\alpha)$ and
 \begin{equation}\label{J/f- eq}
    \left(\frac{J}{f}\right)'= \frac{-1}{Jr^2}\left[f'r^2 \left(\frac{J}{f}\right)^2 -N\left(\frac{J}{f}\right)+1\right].
  \end{equation}
Thus, $\left(\frac{J}{f}\right)'=0$ for
$$\psi_1(s) = \frac{N-\sqrt{N^2-4f'r^2} }{2f'r^2} \quad \mbox{ and } \psi_2(s) = \frac{N+\sqrt{N^2-4f'r^2} }{2f'r^2},  $$
if they exist, and $\left(\frac{J}{f}\right)'>0$ if and only if  $\psi_1< \left(\frac{J}{f}\right)< \psi_2$.
Note that $\psi_1$ and $\psi_2$ depend on the solution, and they exist only when $f'r^2< N^2/4$.

In $\alpha$ we have $\psi_1(\alpha)=\frac{1}{N}= \frac{J}{f}$ and $\psi_2(\alpha)=\infty$.

Also $$\psi_1'(s)= \left(\psi_1-\frac{1}{N}\right) \frac{N}{\sqrt{N^2-4f'r^2}} \left[ \frac{f''}{f'}- \frac{2}{Jr^2}\right],$$
with $\psi_1'(\alpha)= \frac{-2}{N^2}\frac{f'}{f}(\alpha)>\left(\frac{J}{f}\right)'(\alpha).$

\section{ Comparing solutions} \mbox{ } \\

Let $u$ and $v$ be solutions to the initial value problem \eqref{ivp} with $u(0)=u_0> v_0=v(0)$, and let $r_u$ and $r_v$ be the inverses of these solutions. At $v_0$ we have $r_u>0=r_v$, $r'_u>-\infty=r'_v$ and $J_u>\frac{f(v_0)}{N}=J_v$. We want to compare this solutions, and prove that when $f$ is subcritical $J_u(s)>J_v(s)$ until $s=0$, or until their energy $I_u, I_v <0$ and thus they do not reach $0$.\\

To compare these solutions we will use the following Pohozaev type functional, introduced by Erbe and Tang in \cite{et}.

Let
\begin{equation}\label{PdeET}P(s)=r^{N}(s)\left(2N\frac{F}{f}(s)J(s) -r^{2}(s)J^2(s)
-2F(s)\right),\quad \end{equation}
with
\begin{equation*}P'(s)=\frac{d P}{d s}(s)=\left(2N\Bigl(\frac{F}{f}\Bigr)'(s)-(N-2)\right){r^{N}(s)}J(s).
\end{equation*}
Note that if $f$ satisfies $(H_2)$ then $P'(s)>0$ for all $s\in[\beta, \alpha_0]$, and $P(\alpha_0)=0$, thus $P(s)<0$.

We will use these ideas to compare solutions from a point onward, not necessarily form $r=0$,  so we will prove this results for solutions that, at a point $\bar s$, have $r_u>r_v$,  $J_u>J_v$ and $P_u<0\leq P_v$.

\begin{proposition}\label{Prop J comparan}
 Let $f$ be a function that satisfies $(H_1)$ and $(H_2)$, and let $u$ and $v$ be solutions to the initial value problem \eqref{ivp}, that at some $\bar s$ satisfy $r_u(\bar s)>r_v(\bar s)$,  $J_u(\bar s)>J_v(\bar s)$ and $P_u(\bar s)<0\leq P_v(\bar s)$.  Then $r_u'(s)>r_v'(s)$ and $J_u(s)>J_v(s)$  for all $s>-\beta$ with $I_v(s)\geq 0$.

In particular, if  $v$ reaches $0$, then  $r_u'(s)>r_v'(s)$ and $J_u(s)>J_v(s)$ for all $s\geq 0$.
\end{proposition}

Note that, if at some $s_J$ we have $J_u(s_J)=J_v(s_J)$ for the first time ( $J_u>J_v$ in $(s_J, \bar s)$),  then $ \frac{1}{r_u^2}\left(N-\frac{f}{J}\right)=J_u'>J_v'= \frac{1}{r_v^2}\left(N-\frac{f}{J}\right) $ and since $J>\frac{f}{N}$ we have $r_u(s_J)<r_v(s_J)$ and $r_u'(s_J)<r_v'(s_J)$. Since $r_u(\bar s)>r_v(\bar s)$ and $r_u'(\bar s)>r_v'(\bar s)$, there must be  an $s_1$ with $r_u(s_1)=r_v(s_1)$, at which point $r_u'(s_1)>r_v'(s_1)$, and an $s_2$ with $r_u'(s_2)=r_v'(s_2)$, where $s_J<s_2<s_1<\bar s$.
Note also that at $s_2$ we have $I_u(s_2)=I_v(s_2)$.

We will prove this proposition by steps, showing that $s_2$ cannot exist in different intervals, unless $I_v<0$.

\begin{lemma}\label{Lema paso P}
Using the notation above, and the conditions of Proposition \ref{Prop J comparan}, there is no $s_2$ in $[\beta, \bar s]$ and $r^2_uJ_u(\beta)>r_v^2J_v(\beta)$.

\end{lemma}

\begin{proof}
Suppose there is a first (largest) $s_2\in[\beta, \bar s]$ with  $r_u'(s_2)=r_v'(s_2)$, then $J_u(s_2)>J_v(s_2)$. There must be a $\sigma\in [s_2,s_1]$ such that $r_u^2J_u(\sigma)=r_v^2J_v(\sigma)$.  Let $D= \frac{r_u^{N-2}(\sigma)}{r_v^{N-2}(\sigma)}<1$, since $r_u'>r_v'$ in $[s_2, \bar s]$ we have
$r_u^{N-2} \geq D r_v^{N-2}$ in this interval.
Consider the function $(P_u-DP_v)(s)$, it satisfies $(P_u-DP_v)(\bar s)<0$ and by $(H_2)$
\begin{eqnarray*}(P_u-DP_v)'(s)=\left(2N\Bigl(\frac{F}{f}\Bigr)'-(N-2)\right)({r_u^{N}}J_u -D{r_v^{N}}J_v) \\
> \left(2N\Bigl(\frac{F}{f}\Bigr)'-(N-2)\right){r_u^{N-2}}(r_u^2J_u-r_v^2J_v) >0
\end{eqnarray*}
in $[\sigma, v_0]$  therefore  $(P_u-DP_v)(\sigma)<0$ but
\begin{eqnarray*}(P_u-DP_v)(\sigma)= r_u^{N-2}\left(2N\frac{F}{f}(r_u^{2}J_u-r_v^2J_v) -(r_u^{4}J_u^2-r_v^4J_v^2) -2F(s)(r_u^2-r_v^2) \right)\\
=  r_u^{N-2}\left(2F(s)(r_v^2-r_u^2)  \right)>0
\end{eqnarray*}
so we get a contradiction.

Terefore $r^2_uJ_u(s)>r_v^2J_v(s) $  in $[\beta, \bar s]$.

\end{proof}

In the next step we will use the following functional introduced by Peletier and Serrin in \cite{pel-ser1}.
When $I(s)>0$ we can define

\begin{equation}\label{W}W(s)=r(s)\sqrt{r^{2}(s)J^2(s)+2F(s)} = r\sqrt{2I},\quad \end{equation}
with
\begin{equation*}W'(s)=\frac{\partial W}{\partial s}(s)=\frac{(N-2)r^{2}(s)J^2(s)-2F(s)}{r(s)J(s)\sqrt{r^{2}(s)J^2(s)+2F(s)}}.
\end{equation*}
We note that the function $h(s,x)= \frac{(N-2)x^2-2F(s)}{x\sqrt{x^2+2F(s)}}$ is decreasing with $x$ when $F<0$.

\begin{lemma}\label{Lema paso W}
Using the notation above, and the conditions of Proposition \ref{Prop J comparan},  let $s_I\geq -\beta$ be the minimum value where $I_u, I_v\geq 0$.  Then there is no $s_2$ in $[s_I, \beta]$ and if  $ s_I\leq 0$, then  $r_u^2(0)J_u(0) >  r_v^2(0)J_v (0)$. If  $ s_I= -\beta$, then  $r_u^2(-\beta)J_u(-\beta) >  r_v^2(-\beta)J_v (-\beta)$.
\end{lemma}

\begin{proof}
Suppose there is a first (largest) $s_2\in[s_I,\beta]$ with  $r_u'(s_2)=r_v'(s_2)$, then $r_u(s_2)<r_v(s_2)$ and $J_u(s_2)>J_v(s_2)$. Then there must be an $s_1>s_2$ with  $r_u(s_1)=r_v(s_1)$.

 If $s_1\in [\beta, \alpha_0]$ then $r_u(\beta)<r_v(\beta)$ and by Lemma \ref{Lema paso P}
$$W_u(\beta)=r_u^2J_u = \frac{r_u}{|r_u'|}>  \frac{r_v}{|r_v'|}=W_v(\beta).$$

If $s_1\in[s_I,\beta]$, since $J_u(s_1)>J_2(s_1)$,
$$W_u(s_1)=r_u\sqrt{r_u^{2}J_u^2+2F}>r_v\sqrt{r_v^{2}J_v^2+2F}=W_v(s_1). $$

 Let $\bar s = \min\{s_1,\beta\}$, then in $[s_2,\bar s]$ we have $r_uJ_u = \frac{1}{|r_u'|}>  \frac{1}{|r_v'|}=r_vJ_v$ and since $h(s,x)$ decreases with $x$,
$$(W_u-W_v)'(s)=h(s,r_uJ_u)-h(s,r_vJ_v)<0$$
thus $W_u(s_2)>W_v(s_2)$. On the other hand
$$W_u-W_v(s_2)=(r_u-r_v)\sqrt{\frac{1}{|r_u'|^2}+2F(s)}<0$$
and we have a contradiction.

Therefore there is no  $s_2$ in $[s_I, \beta]$ and if $ s_I\leq 0$ then
$$r_u^2J_u (0) = W_v(0)>W_v(0)= r_v^2J_v (0).$$
Similarly, if  $ s_I= -\beta$ then
$r_u^2J_u ( -\beta) = W_v( -\beta)>W_v( -\beta)= r_v^2J_v ( -\beta).$
\end{proof}

\begin{proof}[Proof of Proposition \ref{Prop J comparan}]
Using the notation above, by Lemmas \ref{Lema paso P} and \ref{Lema paso W} there is no $s_2\in[s_I,\bar s]$, and therefore  $r_u'(s)>r_v'(s)$ and $J_u(s)>J_v(s)$  for all $s>s_I$. If $s_I>-\beta$ then $I_u(s_I)= \frac{1}{2 |r_u'(s_I)|^2 }+F(s_I)>\frac{1}{2 |r_v'(s_I)|^2}+F(s_I)=I_v(s_I)=0$ and the first statement follows.

To prove the second statement we note that if $v$ reaches $0$ then $I_v(0)>0$, therefore $ s_I\leq 0$ and the conclusion is obtained for all $s\geq0$.
\end{proof}

\section{The effect of magnitude changes}

We want to study the behaviour of the solutions when $f$ has a magnitude change.
We will approach this problem using a function $f$ defined by parts as in \eqref{fmu} and \eqref{fa}.

In \cite{cghh23} the author in collaboration with C. Cort\'azar and M. Garc\'ia-Huidobro studied this effect by considering functions of the form
 \begin{eqnarray}\label{f lambda}
f(s)=\begin{cases}
f_1(s) &  s\leq \alpha_1\\
L(s) & \alpha_1\leq s\leq \alpha_1+\epsilon\\
\lambda^2 f_2\left(s\right) &  s\geq \alpha_1+\epsilon
\end{cases}
\end{eqnarray}
proving that for appropriate $\alpha_1$, big enough $\lambda$ and small enough $\epsilon$, problem \ref{eq2} has at least two ground state solutions.

We could expect that, if we choose $\alpha_1$ just above a bound state solution with one sign change (a $2^{nd}$-bound state), then for big enough $\lambda$ we will get both a second $2^{nd}$-bound state solution and a second ground state solution. Surprisingly this does not always happen, as we can see in Theorem $B$.

We will begin by understanding the behaviour of $u$ and $J$ below an $\alpha_1$, if it reaches $\alpha_1$ with fixed $|u'|r$ and small enough $r$. We will see that the behaviour will depend strongly on the value $|u'(\alpha_1)|r(\alpha_1)$.

\begin{proposition}\label{Prop J+-}
 Let $f_1$ be a function that satisfies $(H_1)$ and $(H_2)$, and $g_\lambda$, $\lambda\in(1, \infty)$ a family of functions with $g_\lambda(s)=f_1(s)$ for $s<\alpha_1$. Let $u_\lambda$ be a family of solutions to \eqref{ivp} with $f=g_\lambda$   that reach $\alpha_1$ with
 $$\seta <J_\lambda(\alpha_1)r_\lambda^2(\alpha_1)<\ueta<N\frac{f}{f'}(\alpha_1),$$ $$ \quad \lim_{\lambda\to\infty}J_\lambda(\alpha_1)r_\lambda^2(\alpha_1)=\seta \quad \mbox{and }\quad  \lim_{\lambda\to\infty} r_\lambda(\alpha_1)= 0,$$
 where $\seta,\ueta>0$ are constants. For any $\alpha_-, \alpha_+$ such that $\beta<\alpha_-<\alpha_1-\frac{\seta}{N-2}<\alpha_+<\min\{\alpha_1, \alpha_1-\frac{\seta}{N-2}+\frac{2N}{N-2}\frac{F}{f}(\beta)\}$, there exist  $\lambda_0$ and $s_0<\alpha_1$ such that the solutions $u_-$ and $u_+$ with initial conditions $\alpha_-$ and $\alpha_+$ have
$$J_-(s)< J_\lambda (s)< J_+(s) \quad \mbox{ for }\ s<s_0,$$
for all $\lambda>\lambda_0$, as long as their respective energies are positive. \\
\end{proposition}

We will start with two lemmas that will help in the proof, the first one on the behaviour of $\frac{J_\lambda}{f}$.

\begin{lemma}\label{Lema s_lambda} Let  $u_\lambda$ be a family of solutions as in Proposition \ref{Prop J+-},  then there is $\lambda_1$ such that for $\lambda>\lambda_1$ we have $\left(\frac{J_\lambda}{f}\right)'(\alpha_1)>0$ and $\frac{J_\lambda}{f}$ has a first minimum at $s_\lambda$.

 Moreover, $\lim_{\lambda\to\infty}r_\lambda(s_\lambda)=0$,  $\lim_{\lambda\to\infty}\frac{J_\lambda}{f}(s_\lambda)=\frac{1}{N}$ and $\lim_{\lambda\to\infty}s_\lambda=\alpha_1-\frac{\seta}{N-2}$.

\end{lemma}

\begin{proof}
We begin by observing that $J_\lambda(\alpha_1)\to \infty$ with $\lambda$
thus
\begin{equation*}
\left(\frac{J_\lambda}{f}\right)'(\alpha_1)=\frac{1}{fr^2_\lambda}\left[N -\frac{f}{J_\lambda}- J_\lambda r_\lambda^2\frac{f'}{f}  \right]
 >\frac{1}{fr^2_\lambda}\left[\frac{f'}{f}\left(N\frac{f}{f'}- \ueta\right)  -\frac{f}{J_\lambda} \right]
\end{equation*}
will be positive for $\lambda>\lambda_1$ for some $\lambda_1$. Therefore, since $\frac{J_\lambda}{f}$ decreases near $b$, it must have a minimum, and we denote by $s_\lambda$ the largest $s$ where a minimum is obtained.

 Using Lemma \ref{epsilon} with $\bar\alpha+\delta=\alpha_1$ and $\delta< \frac{\seta}{N-2}$ we can find a $B>1$ such that $r_\lambda(\alpha_1-\delta) <B r_\lambda(\alpha_1)$ and
   $$\frac{1}{B^{N-2}} \seta\le J_\lambda r^2_\lambda(\alpha_1-\delta)\le J_\lambda r^2_\lambda(\alpha_1)+\frac{B^N-1}{N}f(\alpha_1)r_\lambda^N(\alpha_1).$$
Therefore $r_\lambda(\alpha_1-\delta) \to 0$ and $J_\lambda r^2_\lambda(\alpha_1-\delta)$ is bounded, thus $J_\lambda(\alpha_1-\delta) \to \infty$.

Using that $(Jr^2)'=(N-2)-\frac{f}{J}$  we can see that
\begin{eqnarray*}
 J_\lambda r^2_\lambda(\alpha_1)-(N-2)\delta  < \  J_\lambda r^2_\lambda(\alpha_1-\delta) &= J_\lambda r^2_\lambda(\alpha_1)-\int_{\alpha_1-\delta}^{\alpha_1} (N-2)-\frac{f}{J_\lambda} \ ds\\ &< J_\lambda r^2_\lambda(\alpha_1)- (N-2)\delta +\frac{f}{J_\lambda}(\alpha_1-\delta)\delta   \end{eqnarray*}
and since  $J_\lambda(\alpha_1-\delta) \to \infty$ and $J_\lambda r^2_\lambda(\alpha_1)\to \seta$ we get that  $J_\lambda r^2_\lambda(\alpha_1-\delta) \to \seta-(N-2)\delta$ for any $\delta< \frac{\seta}{N-2}$.  Given $\epsilon>0$ small , we choose $\alpha_1-\frac{\seta}{N-2}<\alpha_\epsilon <\alpha_1-\frac{\seta-\epsilon}{N-2}$ and $\lambda_2>\lambda_1$ such that $J_\lambda r^2_\lambda(\alpha_\epsilon)<\epsilon$ for all $\lambda>\lambda_2$.

For each solution $u_\lambda$ let $s_\lambda^3$ be the value where $\frac{J}{f}(s_\lambda^3)=\frac{1}{N-3}$, or $s_\lambda^3=s_\lambda$ if $\frac{J}{f}(s_\lambda)>\frac{1}{N-3}$, thus $s_\lambda\leq s_\lambda^3<\alpha_\epsilon$.

In $[s_\lambda^3,\alpha_\epsilon]$,  $\frac{J_\lambda}{f}(s)\geq\frac{1}{N-3}$, therefore $(J_\lambda r_\lambda^2)'(s)=(N-2)-\frac{f}{J_\lambda}>1$ and $$\alpha_\epsilon- s_\lambda^3< J_\lambda r_\lambda^2(\alpha_\epsilon)-J_\lambda r_\lambda^2(s_\lambda^3)< J_\lambda r_\lambda^2(\alpha_\epsilon)<\epsilon.$$  Therefore $\alpha_\epsilon- s_\lambda^3$ and $J_\lambda r_\lambda^2(s_\lambda^3)$ tend to $0$.
Note that since $J_\lambda r_\lambda^2$ tends to $0$ and $J_\lambda\geq f/(N-3)$ is bounded, $r_\lambda(s_\lambda^3)$ tends to $0$. Also, since we can take any $\epsilon$, $ s_\lambda^3$ tends to  $\alpha_1-\frac{\seta-\epsilon}{N-2}$.

If $\lim_{\lambda\to\infty}\frac{J_\lambda}{f}(s_\lambda)\neq\frac{1}{N}$,  let $L^-$ be such that $\frac{J_\lambda}{f}(s_\lambda)>L^->\frac{1}{N}$ for all $\lambda>\lambda_3 $, for some $\lambda_3>\lambda_2$. Note that $L^-<\frac{1}{N-3}$.

Let $\bar\epsilon=N-\frac{1}{L^-}$, if $s_\lambda^3\neq s_\lambda$ then in $[s_\lambda ,s_\lambda^3]$
$$(J_\lambda r_\lambda^{\bar\epsilon})'(s)=r_\lambda^{\bar\epsilon-2}\left((N-\bar\epsilon)-\frac{f}{J_\lambda}\right)\geq 0$$
therefore $J_\lambda r_\lambda^{\bar\epsilon}(s_\lambda)<J_\lambda r_\lambda^{\bar\epsilon}(s_\lambda^3)$ that tends to $0$. Since $J_\lambda(s_\lambda)$ is bounded away from $0$, $r_\lambda(s_\lambda)$ tends to $0$.

Using the analysis of Section \ref{Sec J/f}, we see that for $\lambda>\lambda_0$ we have $\psi_1^\lambda < \frac{J_\lambda}{f}<  \psi_1^\lambda $ and $\frac{J_\lambda}{f}(s)$ will decrease (going backwards) until it reaches $\psi_1^\lambda$. Therefore the minimum will be achieved at an $s_\lambda$ with
$$\frac{J_\lambda}{f}(s_\lambda)=\psi_1^\lambda = \frac{N-\sqrt{N^2-4f'r_\lambda^2} }{2f'r_\lambda^2} . $$
Since  $r_\lambda(s_\lambda)$ tends to $0$, with $f'$ bounded, we have

\begin{equation}\label{1} \lim_  {\lambda\to \infty}\frac{J_\lambda}{f}(s_\lambda)=\lim_{\lambda\to \infty}\frac{N-\sqrt{N^2-4f'r_\lambda^2} }{2f'r_\lambda^2} =\frac{1}{N}, \end{equation}
a contradiction.

Therefore $\lim_{\lambda\to\infty}\frac{J_\lambda}{f}(s_\lambda)=\frac{1}{N}$, and equation \eqref{1} gives that   $\lim_{\lambda\to\infty}r_\lambda(s_\lambda)=0$.

To prove the last limit we note that $(r_\lambda^2)'=-\frac{2}{J}$ is bounded for $s\in [s_\lambda ,s_\lambda^3]$, therefore from the previous statement we get that $s_\lambda^3 -s_\lambda$ tend to $0$, and since $s_\lambda^3\to\alpha_1-\frac{\seta}{N-2}$ we get  $\lim_{\lambda\to\infty}s_\lambda=\alpha_1-\frac{\seta}{N-2}$.

\end{proof}

 To finish the proof of Proposition \ref{Prop J+-} we will use the functional $P$ to compare solutions. For this we recall that $P'>0$ for $s<\alpha_1$, thus $P$ decreases.

\begin{lemma}\label{Lema P+-}
Let  $u_\lambda$ be a family of solutions as in Proposition \ref{Prop J+-},  then for any $\bar\alpha_-, \bar\alpha_+$ such that $\beta<\bar\alpha_-<\alpha_1-\frac{\seta}{N-2}<\bar\alpha_+<\min\{\alpha_1, \alpha_1-\frac{\seta}{N-2}+ \frac{2N}{N-2}\frac{F}{f}(\beta)\}$, there is a $\lambda_3>0$ such that for all $\lambda>\lambda_3$,
$$P_\lambda(\bar\alpha_-)<0 < P_\lambda(\bar\alpha_+).$$
\end{lemma}

\begin{proof}
Let $\alpha_1-\frac{\seta}{N-2}<\bar\alpha_+<\min\{\alpha_1, \alpha_1-\frac{\seta}{N-2}+ \frac{2N}{N-2}\frac{F}{f}(\beta)\}$, by the proof of Lemma \ref{Lema s_lambda}, taking $\bar\alpha_+$ as $\alpha_\epsilon$ for some $\epsilon<2N\frac{F}{f}(\beta)$, when $\lambda\to \infty$,  $r_\lambda(\bar\alpha_+) \to 0$,  $J_\lambda r^2_\lambda(\bar\alpha_+)<\epsilon$ is bounded, and $J_\lambda(\bar\alpha_+) \to \infty$. Therefore
$$P_\lambda(\bar\alpha_+)= r_\lambda^{N}\left( \left(2N\frac{F}{f}-J_\lambda r_\lambda^{2}\right) J_\lambda
-2F\right) $$
will be positive if $J_\lambda(\bar\alpha_+)$ is big enough, therefore there is a $\bar\lambda_3$ such that $P_\lambda(\bar\alpha_+)>0$ for $\lambda>\bar\lambda_3$.

On the other hand, note that
$$\frac{P}{Jr^{N}}=2\frac{F}{f}\left({N}-\frac{f}{J}\right) -Jr^{2} <2\frac{F}{f}\left({N}-\frac{f}{J}\right)$$
has
$$\left(\frac{P}{Jr^{N}}\right)'= 2\left(\frac{F}{f}\right)'\left({N}-\frac{f}{J}\right)
- 2\frac{F}{f}\left(\frac{f}{J}\right)' -(N-2)+\frac{f}{J}$$
$$>\frac{2}{N}\frac{f}{J}- 2\frac{F}{f}\left(\frac{f}{J}\right)' $$

Let $\bar\alpha_-<\alpha_1-\frac{\seta}{N-2}$, then integrating over  $[\bar\alpha_-,s_\lambda]$, where $\left(\frac{f}{J}\right)'<0$, and assuming $\frac{J}{f}<1/(N-1)$ we get
$$\left(\frac{P}{Jr^{N}}\right)(\bar\alpha_-)\leq \left(\frac{P}{Jr^{N}}\right)(s_\lambda)-\frac{2(N-1)}{N}(s_\lambda-\bar\alpha_-)
\leq 2\frac{F}{f}(s_\lambda)\left({N}-\frac{f}{J}(s_\lambda)\right)-\frac{2(N-1)}{N}(s_\lambda-\bar\alpha_-)
$$
that will be negative for $\lambda>\lambda_3$, for some $\lambda_3>\bar\lambda_3$.

\end{proof}

\begin{proof}[Proof of Proposition \ref{Prop J+-}]

Let $s_\lambda$ as in Lemma \ref{Lema s_lambda}, since $\lim_{\lambda\to\infty} r_\lambda(s_\lambda)= 0$ and $s_\lambda\leq \alpha_1-\frac{\seta}{N-2}$, for any $\alpha_1-\frac{\seta}{N-2}<\alpha_+<\min\{\alpha_1, \alpha_1-\frac{\seta}{N-2}+ \frac{2N}{N-2}\frac{F}{f}(\beta)\}$, $u_\lambda$ intersects the solutions $u_+$ at a point $\sigma^\lambda_+$. Moreover, for large enough $\lambda$ we will have $\sigma^\lambda_+>\alpha_1-\frac{\seta}{N-2}$. Let $\lambda_3$ be as in  Lemma \ref{Lema P+-}, and chose $\lambda_0>\lambda_3$ such that  $\sigma^\lambda_+>\alpha_1-\frac{\seta}{N-2}$ for all $\lambda>\lambda_0$.

Then, at $\sigma_+^\lambda$ we have $r_\lambda(\sigma_+^\lambda)=r_+(\sigma_+^\lambda)$ with  $r_\lambda(s)<r_+(s)$ for $s<\sigma_+^\lambda$ near, $r'_\lambda(\sigma_+^\lambda)<r'_+(\sigma_+^\lambda)$, and by Lemma \ref{Lema P+-}   $P_\lambda(\sigma_+^\lambda)>0>P_+(\sigma_+^\lambda)$.

We can use  Proposition \ref{Prop J comparan} to prove $J_\lambda<J_+$ until they reach $0$ or $I<0$.

Let $\beta<\alpha_-<\alpha_1-\frac{\seta}{N-2}$, and $u_-$ the solution with $u_-(0)=\alpha_-$. Then its inverse $ r_-$ satisfies $r_\lambda(\alpha_-)>0=r_-(\alpha_-)$ , $J_\lambda(\alpha_-)<0=J_-(\alpha_-)$ and by Lemma \ref{Lema P+-} $P_\lambda(\alpha_-)<0=P_-(\alpha_-)$ for all $\lambda>\lambda_0$. Therefore we can use Proposition \ref{Prop J comparan} to conclude the proof.

\end{proof}

\begin{corollary}\label{Cor Anti-serrin}

Given $f_1$  that satisfies $(H_1),(H_2)$ and an $\alpha_1>\alpha_*^1$. If there is a $(N-2)(\alpha_1-\alpha_*^1)<K<N\frac{f}{f'}(\alpha_1)$ then there is a $\delta>0$ such that if $u$ is a solution that reaches $\alpha_1$  with $r(\alpha_1)<\delta$ and $|u'(r(\alpha_1))r(\alpha_1)|=K$, then $u(r)>0$ for all $r\in [0, \infty]$.

Note that if $f_1$ satisfies $(H_3)$ there is always such $K$.
\end{corollary}

\begin{proof}
Let $u_\lambda(r)$ be the solution of \ref{ivp} with $u(\frac{1}{\lambda})=\alpha_1$ and $u'(\frac{1}{\lambda})=\lambda K$. Then it satisfies the hypothesis of Proposition \ref{Prop J+-} and we can choose  $\alpha_1- \frac{K}{N-2}<\alpha_+<\alpha_*^1 $ such  that for some $\lambda_0$ and $s_0<\alpha_1$ the solutions satisfy
$$ J_\lambda (s)< J_+(s) \quad \mbox{ for }\ s<s_0,$$
 for all $\lambda>\lambda_0$.
Since $\alpha_+<\alpha_*^1$, $u_+$ cannot reach $0$, and there is an $s_+>0$ with  $J_+(s_+)=0$. Then either $J_\lambda (s)$ reach $0$ or $I(s)<0$ with $s>s_+$. In either case, the solution $u_\lambda$ do not reach $0$.
\end{proof}

\subsection{Proof of the main Theorems}\mbox{}\\

To prove existence of a $k^{th}$-bound state solutions we recall that the sets $\mathcal N_k$ and $\mathcal P_k$ are open subsets of $\mathcal N_{k-1}$, therefore all boundary points must be bound state solutions. If there are $\alpha_N\in \mathcal N_k$ and  $\alpha_P\in \mathcal P_k$, then there must be a boundary point of $\mathcal N_k$ in $(\alpha_P,\alpha_N)$ (or $(\alpha_N,\alpha_P)$). This boundary point must be a bound state solution, and by Lemma \ref{Lema I<0} it must be in $\mathcal N_{k-1}$, thus it is a $k^{th}$-bound state solution.

To prove Theorem A we want to choose $\alpha_1$ and $\epsilon$ in such a way that the solutions with initial condition $\alpha_+$ from Proposition \ref{Prop J+-} do not reach $0$. That is, we need $\alpha_+<\alpha_*^1$. Then, there will be an $\alpha>\alpha_*^k+\epsilon$ with $\alpha\in \mathcal P_1$ and $\alpha_*^k+\epsilon\in \mathcal N_{k+1}$, therefore there must be a $j^{th}$-bound state solution separating $\mathcal N_j$ and $\mathcal P_J$ for each $j\leq k$.  Before choosing $\alpha_+$ we need to control what happens to the solutions when $f=L$.

\begin{lemma}\label{Lema paso e}
Given $\epsilon>0$ and $\alpha_x>\alpha_1+\epsilon$, for each $\lambda$ let $u_\lambda$ be the solution of \ref{ivp} with $f$ as in \eqref{fmu} with constant $\lambda^2$ and $u_\lambda(0)= \alpha_x$, and $r_\lambda$ their respective inverses. Let $v$ be the solution of $\eqref{ivp}$ with $f=f_2$ and $v(0)=\alpha_x$.
If  $\zeta:=J_v r_v^2(\alpha_1+\epsilon)< N\frac{f}{f'} (\alpha_1+\epsilon)$, then $J_\lambda r_\lambda^2(\alpha_1+\epsilon)=\zeta$ is independent of $\lambda$ and $J_\lambda  r_\lambda^2(\alpha_1)\geq \zeta-(N-2)\epsilon$ tends to $\zeta-(N-2)\epsilon$ when $\lambda\to \infty$.

\end{lemma}

\begin{proof}

For each $\lambda$, the function $ {w_\lambda}(r)= {v}(\lambda r)$ satisfies:
$$w_\lambda''(r)+\frac{(N-1)}{r}w_\lambda'(r)= {\lambda^2}\left(v''(\lambda r)+\frac{(N-1)}{\lambda r}v'(\lambda r)\right)= -{\lambda^2}f(w_\lambda(r))$$
therefore $w_\lambda=u_\lambda$ in $[\alpha_1+\epsilon,\alpha_x]$.  Moreover, $r^v= \lambda r_\lambda^u$ and $v'(r^v)= v'(\lambda r_\lambda^u) =  u_\lambda'(r_\lambda^u)/\lambda$ hence $J_\lambda (r_\lambda^u)^2 = r_\lambda^u|u_\lambda'(r_\lambda^u)|=\zeta$ is independent of $\lambda$, and $J_\lambda = \lambda^2J_v$.

Since $\zeta< N\frac{f}{f'} $
\begin{equation*}
\left(\frac{J_\lambda}{f}\right)'(\alpha_1+\epsilon)=\frac{1}{fr^2_\lambda}\left[N -\frac{f}{J_\lambda}- J_\lambda r_\lambda^2\frac{f'}{f}  \right]
 =\frac{1}{fr^2_\lambda}\left[\frac{f'}{f}\left(N\frac{f}{f'}- \zeta\right)  -\frac{f}{\lambda^2J_v} \right]
\end{equation*}
will be positive for $\lambda>\lambda_1$ for some $\lambda_1$.

Using Lemma \ref{epsilon} with $\bar\alpha=\alpha_1+\epsilon$, $\delta=\epsilon$, $g=f$ and $\zeta=J_v r_v^2(\alpha_1+\epsilon)$ we get that for $B>1$ with   $\epsilon \frac{B ^{N-2}}{B ^{N-2}-1}=\frac{J_v r_v^2(\alpha_1+\epsilon)}{N-2}$, $r_\lambda^u (\alpha_1)<B  r_\lambda^u (\alpha_1+\epsilon)$ and
$$\frac{\zeta}{B ^{N-2}}\le J_\lambda (r_\lambda^u)^2 (\alpha_1)\le \zeta+\frac{B^N-1}{N}||g||_+(r_\lambda^u (\alpha_1+\epsilon))^N .$$

Therefore when $\lambda\to\infty$ we get $r_\lambda(\alpha_1) \to 0$ and $J_\lambda r^2_\lambda(\alpha_1)$ is bounded away from $0$, thus $J_\lambda(\alpha_1) \to \infty$.

Using that $(Jr^2)'=(N-2)-\frac{f}{J}$  we can see that
$$  \zeta-(N-2)\epsilon  <  J_\lambda r^2_\lambda(\alpha_1)=\zeta-\int_{\alpha_1+\epsilon}^{\alpha_1} (N-2)-\frac{f}{J_\lambda} \ ds < \zeta- (N-2)\epsilon  +\frac{f}{J_\lambda}(\alpha_1)\epsilon   $$
and since  $J_\lambda(\alpha_1) \to \infty$ and ${f}(\alpha_1)=f_1(\alpha_1)$  we get that  $J_\lambda r^2_\lambda(\alpha_1) \to \zeta-(N-2)\epsilon$.

\end{proof}

\begin{proof}[Proof of Theorem A]\mbox{}\\

Let  $\alpha_1=\alpha_*^k+\epsilon$ and choose  $\alpha_1+ \epsilon - \alpha_*^1< d <\alpha_1+\epsilon - \beta$. Let $v(r)$ be as in condition $(H_5)$, with $v(0)=\widehat \alpha>\alpha_1+ \epsilon$ and let $r_v(s)$ be its inverse. Let $K_m=min_{s\in[0,r_v(\alpha_1+ \epsilon)]} J_vr_v^2 $, and note that $K_m>0$.

For each $\mu>0$, the function $ {w_\mu}(r)= \mu{v}(r)$ satisfies:
$$w_\mu''(r)+\frac{(N-1)}{r}w_\mu'(r)= {\mu}\left(v''( r)+\frac{(N-1)}{ r}v'( r)\right)= -{\mu}f(v)=-{\mu}f(w_\mu/\mu).$$
Let $r_\mu$ be de inverse of $w_\mu$, then $r_\mu(s)=r_v(s/\mu)$ and $J_\mu r_\mu^2(\alpha_1+ \epsilon)=\mu J_v r_v^2(( \alpha_1+ \epsilon)/\mu) > \mu K_m$. Therefore we can choose $\mu_0$ such that for $\mu>\mu_0$ the solution $J_\mu r_\mu^2(\alpha_1+ \epsilon)>(N-2)d$.

For each fixed $\mu$, let $w_\mu(\alpha, r)$ be the solutions with initial condition $w_\mu(\alpha, 0)=\alpha$. Since $J_\mu r_\mu^2(\alpha, \alpha_1+ \epsilon)$ depends continuously on $\alpha$ and  $J_\mu r_\mu^2(\alpha_1+ \epsilon, \alpha_1+ \epsilon)=0$, there must be an $\bar \alpha$ such that $J_\mu r_\mu^2(\bar \alpha, \alpha_1+ \epsilon)=(N-2)d$. By condition $(H_3)$ we have $J_\mu r_\mu^2(\bar \alpha, \alpha_1+ \epsilon)<(N-2)(\alpha_1+ \epsilon-\beta)<N\frac{f}{f'}(\alpha_1+ \epsilon).$

Let $u_\lambda(r)$ be the solution of $\eqref{ivp}$ with $f_\mu$ as in \eqref{fmu} and $u_\lambda(0)=\bar\alpha$. Then by Lemma \ref{Lema paso e}, with $\zeta=  J_\mu r_\mu^2(\bar \alpha, \alpha_1+ \epsilon)$, we have  $J_\lambda  r_\lambda^2(\alpha_1)>\zeta-(N-2)\epsilon$ tends to $\zeta-(N-2)\epsilon$ when $\lambda\to \infty$.

We can now choose $\alpha_+$ with $\alpha_1+ \epsilon-d< \alpha_+<\alpha_*^1$ and use Proposition \ref{Prop J+-} to find $\bar\lambda$ such that the solution $u_+$ with initial condition $\alpha_+$ satisfies
$$J_\lambda (s)< J_+(s) \quad \mbox{ for }\ s<s_0$$
if the energy is positive.

Since $\alpha_+<\alpha_*^1$, $u_+$ cannot reach $0$, and there is an $s_+>0$ with  $J_+(s_+)=0$. Then either $J_\lambda (s)$ reach $0$ or $I(s)<0$ for some $s>s_+$. In either case, the solution $u_\lambda$ do not reach $0$.
By the argument above, there must be a $j^{th}$-bound state solution with initial condition in $(\alpha_*^k+2\epsilon, \bar \alpha)$ for each $j\leq k$.
\end{proof}

\bigskip

To prove Theorem $B$ we will use that for a continuous family of functions, the constants $\lambda_0$ and $s_0$ in Proposition \ref{Prop J+-} can be chosen to depend continuously on the initial condition of the solution.

\begin{lemma}\label{Lema continudad}
 Let $f_1$ be a function that satisfies $(H_1)$ and $(H_2)$, and $g_\lambda$, $\lambda\in(1, \infty)$ a family of functions that depends continuously on $\lambda$, with $g_\lambda(s)=f_1(s)$ for $s<\alpha_1$. Let  $< [\gamma_-,\gamma_+]\subset (\alpha_1,\infty)$  and $u_\lambda(\alpha,r)$ be a family of solutions to \eqref{ivp} with $f=g_\lambda$ that reach $\alpha_1$ with
 $$\seta(\alpha) <J_\lambda(\alpha,\alpha_1)r_\lambda^2(\alpha,\alpha_1)<\ueta(\alpha)<N\frac{f}{f'}(\alpha_1),$$ $$ \quad \lim_{\lambda\to\infty}J_\lambda(\alpha,\alpha_1)r_\lambda^2(\alpha,\alpha_1)=\seta(\alpha) \quad \mbox{and }\quad  \lim_{\lambda\to\infty} r_\lambda(\alpha,\alpha_1)= 0,$$

 For any $\alpha_-, \alpha_+$ such that $$\beta<\alpha_-<\alpha_1-\frac{\sup{\seta(\alpha)}}{N-2}<\alpha_1-\frac{\inf{\seta(\alpha)}}{N-2}<\alpha_+<\alpha_1,$$ there exist  $\lambda_0$ and $s_0<\alpha_1$ such that the solutions $u_-$ and $u_+$ with initial conditions $\alpha_-$ and $\alpha_+$ have
$$J_-(s)< J_\lambda (\alpha,s)< J_+(s) \quad \mbox{ for }\ s<s_0,$$
for all $\lambda>\lambda_0$, and all $\alpha\in(\gamma_-, \gamma_+)$ as long as their respective energies $I(\alpha,s)>0$. \\

\end{lemma}

\begin{proof}
A careful inspection of the proof of Proposition \ref{Prop J+-}, and the lemmas in section $5$, shows that the constants $\lambda_1(\alpha)$, $s_\lambda(\alpha)$, $\sigma_+^\lambda(\alpha)$, etc. are (or can be chosen to be) continuously dependant on $\alpha$.
\end{proof}

\bigskip

\begin{proof}[Proof of Theorem B]\mbox{}\\

We will work with an odd $k$, the even case differ only in some signs. By Lemma \ref{Lema I<0} there is an  $\tilde\alpha^k>\alpha_*^k$ such that all solutions of \eqref{ivp} with $f=f_1$ with $u(0)\in (\alpha_*^k,\tilde\alpha^k]$ reach a minimum with negative energy.  Therefore for each solution there is an $\tilde s\in (-\beta, 0)$ with $J(\tilde s )>0$ and $I(\tilde s)<0$.

It is known that solutions $v(\alpha,r)$ of  \eqref{ivp} with $f=s^p$  and $v(\alpha,0)=\alpha$ converges to a singular solution $\bar v $ on bounded sets when the initial condition $\alpha$ tends to $\infty$.  From  Miyamoto and Naito \cite{mn} we obtain this
for $0<r_0\leq r \leq r_1$ sufficiently small, from continuity of the solutions we can extend this to compact sets. The singular solution $\bar v$ is the classical solution in $(0,\infty)$ such that $\lim\limits_{r\to 0} \bar v(r)=\infty$, (see Serrin and Zou \cite{sz}), it is
$$\bar v(r)=C(N,p) r{^\frac{-2}{p-1}},\quad \mbox{where} \quad  C(N, p)= \Bigl(\frac{2}{p-1}\Bigl(N-2-\frac{2}{p-1}\Bigl)\Bigl)^{1/(p-1)}.$$

If we consider $J_vr^2_v(\alpha, 1)$ for $\alpha>1$, then  $J_vr^2_v(\alpha,1)\to J_{\bar v}r^2_{\bar v}(1)= \frac{2}{p-1} $, therefore it is bounded and there is a $K_1=\sup_{\alpha\in(1,\infty)}J_vr^2_v(\alpha, 1)$. Note also that $v(\alpha, r) = \alpha v(1,\alpha^{\frac{p-1}{2}}r)$ thus $r_v(\alpha, s)= \alpha^{-\frac{p-1}{2}} r_v(1,\frac{s}{\alpha}) $ and $J_vr^2_v(\alpha,s)=\alpha J_vr^2_v(1,\frac{s}{\alpha})$. Therefore $K_s=\sup_{\alpha\in(s,\infty)}J_vr^2_v(\alpha, s)={K_1}{s}$.

 Let $d=\alpha_*^{k+1}+2\epsilon- \frac{\tilde\alpha^k+\alpha_*^k}{2}$,  $\bar s= (N-2)d/K_1$ and set $a=\bar s-(\alpha_*^{k+1}+2\epsilon)$. Then a solution $u(\alpha, r)$ of  \eqref{ivp} with $f=(s+a)^p$ is of the form
$ u(\alpha, r) =v(\alpha+a, r)-a $ and $J_ur_u^2(\alpha, s)=J_vr_v^2(\alpha+a, s+a)$ and
 $J_ur_u^2(\alpha, \alpha_*^{k+1}+2\epsilon)=J_vr_v^2(\alpha+a, \bar s)$
thus
$$K=\sup_{\alpha\in(\alpha_*^{k+1}+2\epsilon,\infty)}J_ur^2_u(\alpha, \alpha_*^{k+1}+2\epsilon)=\sup_{\alpha\in(\bar s,\infty)}J_vr^2_v(\alpha, \bar s)= \bar s K_1= (N-2)d$$

We can now choose $\alpha_+$ and $\alpha_-$ with $\alpha_*^k<\alpha_-<\alpha_*^{k+1}+2\epsilon -d= \frac{\tilde\alpha^k+\alpha_*^k}{2}< \alpha_+<\min\{\tilde\alpha^k,\alpha_*^{k+1}+\epsilon, \alpha_*^{k+1}+\epsilon-\frac{K}{N-2}+ \frac{2N}{N-2}\frac{F}{f}(\beta)\}$.

Let $u$ be any solution of problem \eqref{ivp}, with $f$ as in \eqref{fa} and $\alpha>\alpha_*^{k+1}+2\epsilon $. Then by Lemma \ref{Lema paso e}, with $\zeta= J_ur_u^2 (\alpha_*^{k+1}+2\epsilon)\leq K$, we have  $J_\lambda  r_\lambda^2(\alpha_*^{k+1}+\epsilon)>\zeta-(N-2)\epsilon$ tends to $\zeta-(N-2)\epsilon$ when $\lambda\to \infty$.
Since $\alpha_-< \alpha_*^{k+1}+2\epsilon -d< \alpha_*^{k+1}+\epsilon- \frac{\zeta-(N-2)\epsilon}{N-2}$  we can use Proposition \ref{Prop J+-} to find $\bar \lambda(\alpha)$ and show that the solution $u_-$ with initial condition $\alpha_-$ have
$$J_-(s)< J_\lambda (s) \quad \mbox{ for }\ s<s_0,$$
for all $\lambda>\bar \lambda(\alpha)$, as long as their respective energies are positive. In particular,
$ 0<J_-(0)< J_\lambda (0) $ thus $u$ reaches $0$.

Since when $\alpha \to \infty$ the solutions $u(\alpha,r)$ converge to some $\bar u(r)$ on compact sets, there will be a $\bar \lambda_\infty$ such that we can choose $\bar \lambda (\alpha)<\bar \lambda_\infty$ for $\alpha>A$. Since for $\alpha = \alpha_*^{k+1}+\epsilon$ the solutions are independent of $\lambda$ and satisfy the hypothesis of Proposition \ref{Prop J comparan}, we can choose $\lambda_*$ such that we can choose $\bar \lambda (\alpha)<\bar \lambda_*$ for $\alpha<B$.
From Lemma \ref{Lema continudad} the choice of $\bar\lambda(\alpha)$ can be done uniformly for $\alpha\in[B,A]$. Therefore choosing $\bar \lambda$ as de largest of these three we get that for $\lambda>\bar \lambda$ al solutions with initial condition $\alpha>\alpha_*^{k+1}+\epsilon$ cut $0$ at least once.

On the other hand, since $K$ is a supremum, there is an $\alpha_\Box$ such that  $J_ur^2_u(\alpha_\Box, \alpha_*^{k+1}+2\epsilon)> (N-2)( \alpha_*^{k+1}+2\epsilon- \tilde\alpha^1)$.
 Then by Lemma \ref{Lema paso e}, and Proposition \ref{Prop J+-} we can find
$\bar \lambda(\alpha_\Box)$ and show that the solution $u_+$ with initial condition $\alpha_+$ have
$$J_+(s)> J_\lambda (s) \quad \mbox{ for }\ s<s_0,$$
for all $\lambda>\lambda_0$, as long as their respective energies are positive. In particular,
$ J_\lambda (\tilde s)< J_+(\tilde s)<0 $ thus $u$ can not reach $0$ a second time.
By the argument at the beginning of this subsection, there must be a $k^{th}$-bound state solution with initial condition in $(\alpha_*^k+2\epsilon,\alpha _\Box)$.

\end{proof}


\newpage
\begin{thebibliography}{AAAA}
	
\bibitem[BN]{bn}{\sc H. Berestycki and L. Nirenberg,} Monotonicity, symmetry and antisymmetry of solutions of semilinear elliptic equations, {\em J. Geometry and Physics,}  {\bf 5} (1988), 237--275.

\bibitem[CGHH1]{cghh15}{\sc Cort\'azar, C., Garc\'\i a-Huidobro, M., Herreros, P. } Multiplicity results for sign changing bound state solutions of a semilinear equation. {\em J. Differential Equations} {\bf 259 }(2015), no. 12, 7108--7134.
	
\bibitem[CGHH2]{cghh23} {\sc Cort\'azar, C., Garc\'\i a-Huidobro, M., Herreros, P. } Multiplicity results for ground state solutions of a semilinear equation via abrupt changes in magnitude of the nonlinearity,{\em  Discrete And Continuous Dynamical Systems} {\bf 42 } (2023)

\bibitem[ET]{et} {\sc  Erbe, L.,  Tang, M., } Uniqueness theorems for
positive solutions of quasilinear elliptic equations in a ball, {\it
	J. Diff. Equations} {\bf 138} (1997), 351-379.


\bibitem[FL]{fl}{\sc Franchi, B., Lanconelli, E., }Radial symmetry of the ground states
for a class of quasilinear elliptic equations, {\em Nonlinear Diffusion Equations and their equilibrium states} Vol. 1 (1988), 287--292.


\bibitem[FLS]{fls}{\sc Franchi, B., Lanconelli, E., Serrin, J., }Existence and Uniqueness of nonnegative
solutions of quasilinear equations in $\RR^n $, {\em Advances in mathematics} {\bf 118 }(1996), 177--243.


\bibitem[GNN]{gnn}{\sc Gidas, B.,  Ni, W. M., Nirenberg, L.}, {\em Symmetry of positive solutions of nonlinear elliptic equations in $\RR^n$},Mathematical analysis and applications, Part A, pp. 369 –- 402, Adv. in Math. Suppl. Stud., 7a, Academic Press, New York-London, 1981.

	
\bibitem[GST]{gst}{\sc Gazzola, F., Serrin, J. and Tang, M., }Existence of ground states and free boundary value problems for quasilinear elliptic operators. {\em Advances in Diff. Equat.} {\bf 5 }(2000), no. 1-3, 1-30.
	
\bibitem[LN]{ln}{\sc Li, Y., Ni, W}. Radial Symmetry of Positive Solutions of Nonlinear Elliptic Equations in $\R^n$,
{\em Communications in Partial Differential Equations,} {\bf 18}(1993) 1043--1054.

\bibitem[MN]{mn}{\sc Miyamoto, Yasuhito; Naito, Yūki}. Fundamental properties and asymptotic shapes of the singular and classical radial solutions for supercritical semilinear elliptic equations. {\em NoDEA Nonlinear Differential Equations Appl.} {\bf 27 }(2020), no. 6, Paper No. 52, 25 pp.	
	

\bibitem[PS]{pel-ser1}{\sc Peletier, L., Serrin, J., }Uniqueness of positive
solutions of quasilinear equations, {\em Archive Rat. Mech. Anal.} {\bf 81 }(1983), 181-197.
	
\bibitem[SZ]{sz}{\sc Serrin, J., Zou, H.} Classification of positive solutions of quasilinear elliptic equations.
{\em Topol. Methods Nonlinear Anal.} {\bf 3 }(1994), 1–25.	
\end{thebibliography}
\end{document}